\documentclass[12pt, reqno]{amsart}
\usepackage{amsmath, amsthm, amsxtra, amscd, amsfonts, amssymb, mathrsfs, graphicx, color,ulem, pstricks}

\usepackage[bookmarksnumbered, colorlinks, plainpages]{hyperref}
\hypersetup{colorlinks=true,linkcolor=red, anchorcolor=green, citecolor=cyan, urlcolor=red, filecolor=magenta, pdftoolbar=true}
\usepackage{setspace}

\theoremstyle{plain}
\newtheorem{theorem}{Theorem}[section]
\newtheorem{lemma}[theorem]{Lemma}
\newtheorem{proposition}[theorem]{Proposition}
\newtheorem{corollary}[theorem]{Corollary}
\theoremstyle{definition}

\newtheorem{example}[theorem]{Example}

\theoremstyle{remark}

\numberwithin{equation}{section}

\definecolor{darkgreen}{rgb}{.1,.5,0}

\theoremstyle{plain}

\theoremstyle{definition}

\newtheorem*{Sketch of proof}{Sketch of proof}

\numberwithin{equation}{section}
\begin{document}
\setcounter{page}{1}
\title[Extension of $m$-isometric  weighted composition operators on directed graphs\textrm{~} ]{ Extension of $m$-isometric  weighted composition operators on directed graphs }

\author[V. Devadas, E. Shine Lal   and  T. Prasad ]{V. Devadas, E. Shine Lal   and  T. Prasad}

\address{V. Devadas \endgraf
		Department of Mathematics, Sree Narayana College,  Alathur, 
		Affiliated to University of Calicut,
		Kerala, 
		India -678682}

	\email{\textcolor[rgb]{0.00,0.00,0.84}{ v.devadas.v@gmail.com}}
\address{E. Shine Lal  \endgraf
		Department of Mathematics, University College, Thiruvananthapuram,
		Kerala, 
		India- 695034.}

	\email{\textcolor[rgb]{0.00,0.00,0.84}{ shinelal.e@gmail.com}}

\address{T. Prasad\endgraf
		Department of Mathematics, 
		University of Calicut,
		Kerala-673635, 
		India.}
	
	\email{\textcolor[rgb]{0.00,0.00,0.84}{ prasadvalapil@gmail.com}}


\begin{abstract}
In this paper, we discuss $k$-quasi-$m$-isometric  composition operators  and weighted composition operators  on  directed graphs with one circuit and more than one branching vertex.
\end{abstract}
\maketitle

\section{Introduction and Preliminaries}
Let  $B(\mathcal{H}) $ denote the algebra of all bounded linear operators  on  a complex Hilbert space $\mathcal{H}$. The symbols $\mathbb{N},\mathbb{Z_+},\mathbb{Z},\mathbb{R},$ and $\mathbb{C}$ stand for the set of  natural numbers, nonnegative integers, integers, real numbers and complex numbers respectively. For $T\in B(\mathcal{H})$ and for  $m\in \mathbb{Z_+} $,  define
 $$ \mathcal{B}_m(T)=\sum^{m}_{j=0} (-1)^{j}\begin{pmatrix} m \\
	j \end{pmatrix}T^{* (m-j)}T^{(m-j)},$$  where $T^*$ stands for adjoint of $T$ and 
 $\begin{pmatrix} m \\
 j \end{pmatrix}$
 the binomial coefficient.  For $m\in\mathbb{N}$, an operator  $T\in B(\mathcal{H})$ is said to be  $m$-isometric if $ \mathcal{B}_m(T)=0$  \cite{as1,as2,as3}. For $k,m \in \mathbb{N}$, an operator $T \in B(H)$ is said to be   $k$-quasi-$m$-isometric if 
 $ T^{*k}\mathcal{B}_m(T)T^k=0$ \cite{Mechpra}.  Class of $m$-isometric operators and related classes  has been studied extensively  (see  \cite{m1,ZJIBJS,m2,SP,m3, PT,Rich}).

Let  $(X,  \mathcal{F}, \mu)$ be the discrete measure space, where  $X$ is a countably infinite set  and $\mu$ is  a positive measure on $ \mathcal{F} $, the $\sigma$- algebra of all subsets of $X$ such that $\mu(\{x\})\geq 0$ for every $x\in X$. A measurable function $\phi$ from $X$ into itself means $\phi^{-1}(\mathcal{F})\subset\mathcal{F}$. Note that  the measure  $\mu \circ \phi^{-1}$ on $\mathcal{F}$  is given by $\mu \circ \phi^{-1}(S)= \mu (\phi^{-1}(S)) ~~~\textrm{for all}~~~ S\in \mathcal{F} $. 
  Recall that if  $\mu \circ \phi^{-1}$ is absolutely continuous with respect to $\mu$,  we call the map $\phi$ is nonsingular. Then the Radon -Nikodym derivative of  $\mu \circ \phi^{- 1}$ with respect to $\mu$ exists and is denoted by $h$.  We know that  if $\phi$ is nonsingular, then $\phi^{p}$ is nonsingular for every  $p\in \mathbb{Z_+} $. In this case, Radon -Nikodym derivative of $\mu \circ\phi^{-p}$ with respect to $\mu$ is denoted by $h_p$. In particular $h_0=1$ and $h_1=h$.
  
   Let $L^{2}(X,  \mathcal{F}, \mu) (= L^{2}(\mu))$ be the space of all equivalence classes of square integrable complex valued functions on $X$ with respect to the measure $\mu$. Then the   composition operator $C$ on $L^{2}(\mu)$ induced by  a nonsingular measurable transformation $\phi$ on $X$ is given  by $C f= (f \circ \phi)$, $~f \in L^{2}(\mu)$.  Composition operator $C$ is bounded if and only if  the Radon -Nikodym derivative $h$ is essentially bounded. In this case $\parallel C _{\phi} \parallel^2 = \parallel h \parallel_{\infty}$ and 
   $ \parallel C^n(f) \parallel^2= \int_{S}^{} h_n|f|^2 d\mu, ~~ f\in L^{2}(\mu), n\in \mathbb{Z_+} .$
   
   Let $L^{\infty}(\mu)$ be the space of all equivalence classes of essentially bonded and  measurable complex valued functions on $X$ with respect to the measure $\mu$. If  $\pi \in  L^{\infty}(\mu)$ and $\phi$ is a nonsingular measurable transformation $\phi$ on $X$ . Then the  multiplication operator  $M$ induced by $\pi$ on $L^{2}(\mu)$ is given by
$M_{\pi} f=\pi  f,$  $~f \in L^{2}(\mu)$. The weighted composition operator $W$  on $L^{2}(\mu)$ induced by a nonsingular measurable function  $\phi$ and an  essentially bounded function $\pi$ is given by
$Wf= \pi (f \circ \phi),  ~ f \in L^{2}( \mu)$. 
Let  $\pi_{k} = \pi(\pi \circ \phi)(\pi \circ \phi^{2}).....(\pi \circ \phi^{k-1})$, $k\in \mathbb{N}$. Then we have   $W^{k}f=\pi_{k}(f \circ \phi)^{k}, ~ f\in L^{2}( \mu)$.   General properties of  composition operators has been found in  \cite{nordgren, SM }.

  If $\phi$ is a nonsingular measurable function, then $\phi^{-1}\mathcal{F}$ is a $\sigma$-subalgebra of $\mathcal{F}$ and $L^2(X,\phi^{-1}\mathcal{F},\mu)$ is a closed subspace of the Hilbert space $L^2(X,\mathcal{F},\mu) $. The conditional expectation operator associated  with $\phi^{-1}\mathcal{F}$ is an  orthogonal projection of $L^2(X,\mathcal{F},\mu)$ onto $L^2(X,\phi^{-1}\mathcal{F},\mu)$  defined for all non-negative measurable functions  $f$ on $X$ and  $f\in L^2(X,\mathcal{F},\mu) $. For each $f$ in the domain of $E$, $E(f)$ is the unique $\phi ^{-1}\mathcal{F}$ measurable  function satisfying 
  $$ \int_{S} f d\mu=\int_{S}E(f) d\mu, ~~~~\text{for all }  S\in \phi^{-1}\mathcal{F}.$$
   We denote the conditional expectation associated with $\phi^{-n}\mathcal{F}$  by $E_n$. If $\phi^{-n}\mathcal{F}$ is purely atomic $\sigma$-subalgebra of $\mathcal{F}$ generated by the atoms $\{ A_k\}_{k\geq 0}$, then
  $$ E_n(f|\phi^{-n}\mathcal{F})=\sum_{k=0}^{\infty} \frac{1}{\mu(A_k)}\displaystyle \left(\int_{A_k} fd\mu\right)\chi_{A_k}.$$
  We refer the reader to \cite{ca:ja, Herron, la, Rao} for more details on the  properties of conditional expectation.
  
The study of weighted shift operators on directed trees by Jabłoński, Jung, and Stochel\cite{JJJ}  has been a stimulation for  the  study of the  classes  of non-normal operators in the view point of composition operators and weighted shift operators on directed graph settings (see \cite{g2,g1, g3,g4,ZJJK ,g5,g6} ).  Recently,  Jabłoński and Kośmider \cite{ZJJK } characterized  $m$--isometric composition operators on directed graphs with one circuit. 
In this paper, we  characterize  $k$-quasi-$m$-isometric   composition operators on $L^2(\mu)$ with respect to the positive measure $\mu$ on directed graphs with one circuit and more than one branching vertex influenced by the treatment of  Jabłoński and Kośmider \cite{ZJJK }, and we study  $k$-quasi-$m$-isometric   weighted composition operators on $L^2(\mu)$ with respect to the positive measure $\mu$ on directed graphs with one circuit and more than one branching vertex.

%%%%%%%%%%%%%%%%%%%%%%%%%%%%%%%%%%%%%%%%%%%%%%%

\section{$k$-quasi-$m$-isometric  composition  and weighted composition operators} \label{sn2} 
 Let   $J_\kappa=\{1,2,\ldots,\kappa\}$,  $\kappa\in \mathbb{N}$ and let $\eta_r\in\mathbb{Z_+\cup\{\infty\}}$, $r\in J_\kappa$.  Suppose that  at least one of $\eta_r$ is non-zero for $r\in J_\kappa$ and

\begin{equation} \label{eqn1}
	\nonumber X=\{x_1,x_2,\ldots,x_k\}\cup\bigcup_{r=1}^{\kappa}\bigcup_{i=1}^{\eta_r} \{x^r_{i,j}: j\in\mathbb{N}\},
\end{equation}
where 
$X_\kappa=\{x_1,x_2,\ldots,x_k\}$ and $X_{\eta_r}=\bigcup_{i=1}^{\eta_r} \{x^r_{i,j}: j\in\mathbb{N}\}$ ($r\in J_\kappa$)  are  disjoint sets of distinct points of $X$.  Throughout this section we consider $X$ as a  directed graph with one circuit $\{x_1,x_2,\ldots,x_k\}$, the set of  branching vertices in the one-circuit   
and  $X_{\eta_r}$,  the set of branching elements for $r\in J_\kappa$ where $\{x^r_{i,j}: j\in\mathbb{N}\}$ is the set of all vertices in the $i^{th} $ branch of  $x_r$ for $i\in J_{\eta_r}$ and  $\eta_r$ is the number of branches originating  from the vertex  $x_r$.  Recently,  a general version of  this type of  graph  has been considered by  Buchała\cite{MB} . The following figure 1 represent the  above discussed graph for the case $\kappa = 3$ and $\eta_{r}=2, r\in J_{\kappa}$.
\begin{figure}[h]
	\centering
	\includegraphics[width=1.1\textwidth]{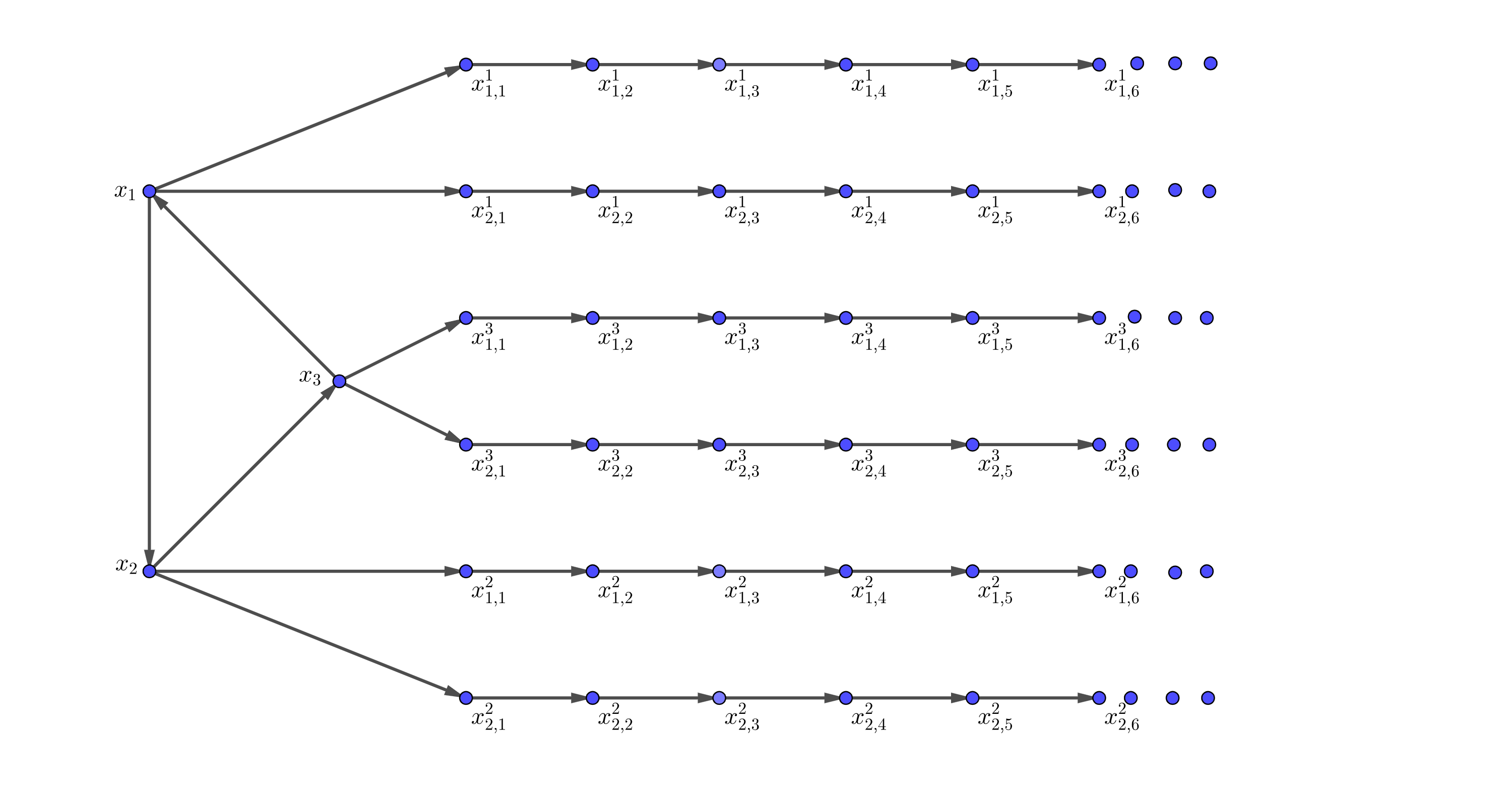} 
	\caption{Directed graph with one circuit and more than one branching vertex}
\end{figure}

\label{fig:Directed graph with one circuit}
Consider $(X, \mathcal{F}, \mu)$ as a $\sigma$-finite measure space, where $\mu$ is a $\sigma$-finite positive measure on $X$ with $\mu({x}) > 0$ for every $x \in X$. We will use  the following functions $\Phi_1$ and $\Phi_2$ to define the parent function on $(X, \mathcal{F}, \mu)$, which will assist for  determine  the atoms of the $\sigma$-algebra $\phi^{-p}(\mathcal{F})$ within $\mathcal{F}$. Let $\kappa\in \mathbb{N}$, and let $\Phi_1: \mathbb{Z}\rightarrow \mathbb{Z}$ and $\Phi_2: \mathbb{Z}\rightarrow J_\kappa$ be two uniquely determined functions defined by 
	$p=\Phi_1(p)\kappa+ \Phi_2(p),~~p\in\mathbb{Z}$. These functions satisfies  the conditions:
	$\Phi_1(l\kappa+1)=\Phi_1(l\kappa+r), ~~~~~l\in \mathbb{Z},~~r\in J_\kappa$, and 
	$\Phi_2(l\kappa+r_1+r_2)=\Phi_2(l\kappa+r_1)+r_2, ~~~~~l\in \mathbb{Z},$ for
	$r_1\in \mathbb{N}, ~~r_2\in \mathbb{Z_+},~~r_1+r_2\in J_\kappa.$ From the above  directed graph, we obtain the parent function as follows:
	
\begin{align*}
par(x)= \left\{
\begin{array}{ll}\
	x^r_{i,j} & \mathrm{if}~~ x=x^r_{i,j+1} ~~~\mathrm{for}~~r\in J_\kappa~~~ i\in J_{\eta_r},~~ \mathrm{and}~~j\in \mathbb{N} ,\\\\
	x_r & \mathrm{if}~~ x=x^s_{i,j} ~~~\mathrm{for}~~ s\in J_\kappa,  \mathrm{and}~~ \Phi_2(1+r)=\Phi_2(s+j),~~  j\in J_1, i\in J_{\eta_s},\\
	 & ~~ \mathrm{or}~~x=x_{\Phi_2(1+r)}.
\end{array}\right. \end{align*}
\label{eqn1}

\begin{align}\label{equ1} 
		\textrm{ Assume that }(X, \mathcal{F}, \mu)\textrm{ is a discrete measure space where X is a directed }~~ \nonumber \\ \textrm{  graph with one ~~circuit and~~ more than ~~one branching vertex as discussed~~~ above  } ~~ \nonumber \\ 
		\textrm{and~~}\phi~~\textrm{ is a~~measurable~~ transformation~~   on}~~ X		   
	\textrm{~~defined~~~ by~~ ~} \phi(x)=par(x),~~~~~~x\in X 
\end{align}

From the functions $\Phi_{1}$ and $\Phi_{2}$ discussed earlier, we  derive the general expression for the $p$-fold of $\phi$ as follows:\\
$$\phi^p(x)= \left\{
\begin{array}{ll}\label{AS}
	x^r_{i,j} & \mathrm{if}~~ x=x^r_{i,j+p} ~~~\mathrm{for}~~r\in J_\kappa~~~ i\in J_{\eta_r},~~ \mathrm{and}~~j\in \mathbb{N} ,\\\\
	x_r& \mathrm{if}~~ x=x^s_{i,j} ~~~\mathrm{for}~~ s\in J_\kappa,  \mathrm{and}~~ \Phi_2(p+r)=\Phi_2(s+j),~~  j\in J_p, i\in J_{\eta_s},\\
	& ~~ \mathrm{or}~~x=x_{\Phi_2(p+r)}.\\\\
\end{array}\right.$$

Hence, the atoms of the $\sigma$-algebra $\phi^{-p}(\mathcal{F})$ within $\mathcal{F}$ can be determined as follows:

$$\phi^{-p}(\{x\})= \left\{
\begin{array}{ll}
	\{x^r_{i,j+p}\}& \mathrm{if}~~ x=x^r_{i,j} ~~ r\in J_\kappa , \\
	& ~~i\in J_{\eta_r},~~j\in \mathbb{N},\\\\
	\{x_{{\Phi_2}_(p+r)}\}\cup\bigcup_{j=1}^p\bigcup_{s=1,   \Phi_2(p+r)=\Phi_2(s+j)}^\kappa\bigcup_{i=1}^{\eta_s}
	 \{x^s_{i, j}\} &\mathrm{if}~~ x=x_r,~~r\in J_\kappa .\\\\
\end{array}\right.$$

Given that $\mu(x)>0$ for every $x \in X$ and the transformation $\phi$ is nonsingular. Cconsequently, $\phi^p$ is also nonsingular for $p \in \mathbb{N}$. Therefore, the Radon-Nikodym derivative $h_p = \frac{d(\mu \circ \phi^{-p})}{d\mu}$ can be determined using the atoms of the $\sigma$-algebra $\phi^{-p}(\mathcal{F})$ as follows:
$$h_p(x)= \left\{
\begin{array}{ll}
	\frac{\mu(x^r_{i,j+p})}{\mu(x^r_{i, j})} & \mathrm{if}~~ x=x^r_{i,j},  ~~ r\in J_\kappa,  ~~  i\in J_{\eta_r}, \\
	 & ~~j\in \mathbb{N},\\\\
\frac{\mu(x_{{\Phi_2}_(p+r)})+\Sigma_{j=1}^p\Sigma_{s=1, \Phi_2(p+r)=\Phi_2(s+j) }^{\kappa}
	\Sigma _{i=1}^{\eta_{s}} \mu(x^s_{i,j})}{\mu(\{x_r\})}	 &\mathrm{if}~~ x=x_r,~~r\in J_\kappa .\\\\
\end{array}\right.$$

%%%%%%%%%%%%%%%%%%%%%%%%%%%%%%%%%%%%%%%%%%%%%%%%%%%%%%%%

Recall the following result by  Jabłoński, Jung, and Stochel\cite{ZJIBJS}. 
Consider $\mathbb{R^{Z_+}}$ as the space of all real-valued sequences indexed by $\mathbb{Z_+}$, and $\mathbb{R}[x]$ as the ring of polynomials in $x$ with real coefficients. A sequence $\gamma = \{\gamma_n\}_{n=0}^{\infty}$ in $\mathbb{R^{Z+}}$ is said to be a polynomial of degree $k \in \mathbb{Z_+}$ if there exists a polynomial $p(x) \in \mathbb{R}[x]$ of degree $k$ such that $p(n) = \gamma_n$ for all $n \in \mathbb{Z_+}$.  
For $~~~m,n\in\mathbb{Z_+},\gamma=\{\gamma_n\}_{n=0}^{\infty}\in \mathbb{R^{Z_+}},$  define  an operator  $\Delta$  on $\mathbb{R^{Z_+}}$ by
$(\Delta\gamma)_n= \gamma_{n+1}-\gamma_n$. Then,  $(\Delta^m\gamma)_n= (-1)^m \sum^{m}_{k=0} (-1)^{k}\begin{pmatrix} m \\
	k \end{pmatrix}\gamma_{n+k}$ (\cite{ZJIBJS}).
	 %%%%%%%%%%%%%%%%%%%%%%%%%%%%%%%%%%%%%%%%%%%%%%%%%5%%%%
\begin{lemma}\label{l1}(\cite{ZJIBJS}) \label{l1} Let $m\in\mathbb{N}$ and  $\gamma=\{\gamma_n\}_{n=0}^{\infty}\in \mathbb{R^{Z_+}}$. Then the following are equivalent:\\\\
(i) $\Delta^m\gamma=0,$\\
(ii) $\sum^{m}_{j=0} (-1)^{j}\begin{pmatrix} m \\\\
	j \end{pmatrix}\gamma_{n+j}=0~~~\mathrm{for}~~~n\in\mathbb{Z_+} ,$ \\
(iii) $\gamma_n$ is a polynomial in $n$ of degree at most $m-1$.
\end{lemma}

The following result is immediate from Lemma \ref{l1} and the generalization of  \cite[Theorem 2.2]{SPD } for $k$-quasi-$m$-isometric composition operators.
\begin{lemma} \label{l4}Let  $(X,  \mathcal{F}, \mu)$ be a discrete measure space, $\phi$ be a nonsingular measurable transformation on $X$,  and   $C$ be the  composition operator on $L^{2}(\mu)$ induced by  $\phi$. Then for any $m\in\mathbb{N}$ and $k\in\mathbb{Z_+}$ the following are equivalent:\\\\
	(i) $C$ is an $k$-quasi-$m$-isometry,\\\\
	(ii) $\sum^{m}_{j=0} (-1)^{j}\begin{pmatrix} m \\
		j \end{pmatrix} C^{*(k+j)}C^{(k+j)}=0,$\\\\
	(iii) $\sum^{m}_{j=0} (-1)^{j}\begin{pmatrix} m \\
		j \end{pmatrix} C^{*(n+k+j)}C^{(n+k+j)}=0~~~\mathrm{for}~~~n\in\mathbb{Z_+},$\\\\
	(iv) $\sum^{m}_{j=0} (-1)^{j}\begin{pmatrix} m \\
		j \end{pmatrix} h_{n+k+j}(x) =0~~~\textrm{for all}~~~ x\in X~~~\mathrm{and}~~~n\in\mathbb{Z_+},$\\\\
	(v) $\{h_{n+k}(x)\}_{n=0}^\infty$ is a polynomial in $n$ of degree at most $m-1$ for all $x \in X.$	\\
\end{lemma}

%%%%%%%%%%%%%%%%%%%%%%%%%%%%%%%%%%%%%%%%%%%%%%%%%%%

\begin{lemma}\label{l5}
	Let $p, \kappa \in \mathbb{N}$, $k\in\mathbb{Z_+}$ and $r\in J_{\kappa}$. If 
$$ A_r= \{(s,j)/ s\in J_\kappa, j\in J_{p+k}, \Phi_2(p+k+r)=\Phi_2(s+j)\},$$
then $\{A_r\}_{r \in J_\kappa}$ form a partition of the set  $ A= \{ (s, j)/ s\in J_\kappa, j\in J_{p+k}\}$.
	\end{lemma}
	
 \begin{proof}
 	First note that each $A_r$ is nonempty.  Since  $ \Phi_2(p+k+1),~~~~~~ \Phi_2(p+k+2),\ldots, \Phi_2(p+k+\kappa)$ are the distinct elements in the set $J_\kappa$, 
$A_r \cap A_t$ is empty for $s\neq t$ in $J_\kappa$.  If $(s, J)\in A$, then there exists $r\in J_\kappa$ such that $(s, J)\in A_r $ since $ 2\leq s+j \leq p+k+\kappa. $ Therefore, $ A = \bigcup_{r \in J_\kappa} A_r $.
 \end{proof}
 
 %%%%%%%%%%%%%%%%%%%%%%%%%%%%%%%%%%%%%%%%%%%%%%%%%%%%%%%%%
\begin{lemma}\label{l6}
	Assume that $m\in \mathbb{N}$, $k\in\mathbb{Z_+}$, (\ref{equ1}) holds and   $$ \Sigma_{r=1}^\kappa \Sigma_{i=1}^{\eta_r} \Sigma_{j=1}^m \mu(x_{i, j}^r) <\infty.$$ Then 
	$$
	\begin{array}{ll}
		\Sigma_{r=1}^\kappa \mu(x_r) \Sigma_{p=0}^{m} (-1)^p \begin{pmatrix} m \\
			p \end{pmatrix} h_{p+k}(x_r)  &=
		-\Sigma_{r=1}^{\kappa} \Sigma_{i=1}^{\eta_r} \Sigma_{p=0}^{m-1} (-1)^p \begin{pmatrix} m-1 \\ p \end{pmatrix} \mu(x_{i, p+k+1}^r) \\
	& = -\Sigma_{r=1}^{\kappa} \Sigma_{i=1}^{\eta_r} \Sigma_{p=0}^{m-1} (-1)^p \begin{pmatrix} m-1 \\ p \end{pmatrix} \mu(x_{i, 1}^r)h_{p+k}(x_{i, 1}^r).
	\end{array} $$
\end{lemma}
\begin{proof}
	Let $p\in \mathbb{N}$ and $k\in\mathbb{Z_+}$. The  Radon Nikodym derivative $h_{p+k}$  is defined by 
	
	$$h_{p+k}(x)= \left\{
	\begin{array}{ll}
		\frac{\mu(x^r_{i,j+p+k})}{\mu(x^r_{i, j})}& \mathrm{if}~~ x=x^r_{i,j},  ~~ r\in J_\kappa,  ~~  i\in J_{\eta_r},\\
		& ~~j\in \mathbb{N},\\\\
		\frac{\mu(x_{{\Phi_2}_(p+k+r)})+\Sigma_{j=1}^{p+k}\Sigma_{s=1,  \Phi_2(p+k+r)=\Phi_2(s+j) }^{\kappa}
			\Sigma _{i=1}^{\eta_{s}} \mu(x^s_{i,j})}{\mu(x_r)}	 &\mathrm{if}~~ x=x_r,~~r\in J_\kappa. \\\\	
	\end{array}\right.$$
	
Now we obtain\\
	$\Sigma_{r=1}^\kappa \mu(x_r) \Sigma_{p=0}^{m} (-1)^p \begin{pmatrix} m \\
		p \end{pmatrix} h_{p+k}(x_r) $\\
	$ = \Sigma_{r=1}^\kappa \mu(x_r) \Sigma_{p=0}^{m} (-1)^p \begin{pmatrix} m \\
		p \end{pmatrix}\frac{\mu(x_{{\Phi_2}_(p+k+r)})+\Sigma_{j=1}^{p+k}\Sigma_{s=1, \Phi_2(p+k+r)=\Phi_2(s+j) }^{\kappa}
		\Sigma _{i=1}^{\eta_{s}} \mu(x^s_{i,j})}{\mu(x_r)}$ \\
	$ = \Sigma_{r=1}^\kappa  \Sigma_{p=0}^{m} (-1)^p \begin{pmatrix} m \\
		p \end{pmatrix}\mu(x_{{\Phi_2}_(p+k+r)}) 	\\
	 + \Sigma_{r=1}^\kappa  \Sigma_{p=0}^{m} (-1)^p \begin{pmatrix} m \\
		p \end{pmatrix} \Sigma_{j=1}^{p+k}\Sigma_{s=1, \Phi_2(p+k+r)=\Phi_2(s+j) }^{\kappa}
		\Sigma _{i=1}^{\eta_{s}} \mu(x^s_{i,j})$\\
	$ = 0	+ \Sigma_{r=1}^\kappa  \Sigma_{p=0}^{m} (-1)^p \begin{pmatrix} m \\
		p \end{pmatrix} \Sigma_{j=1}^{p+k}\Sigma_{s=1, \Phi_2(p+k+r)=\Phi_2(s+j) }^{\kappa}
	\Sigma _{i=1}^{\eta_{s}} \mu(x^s_{i,j}) $ \\
	$=   \Sigma_{p=0}^{m} (-1)^p \begin{pmatrix} m \\
		p \end{pmatrix} \Sigma_{j=1}^{p+k}\Sigma_{r=1}^{\kappa}
	\Sigma _{i=1}^{\eta_{r}} \mu(x^r_{i,j}) $\\
	$= \Sigma_{p=0}^{m} (-1)^p \begin{pmatrix} m \\
		p \end{pmatrix} \Sigma_{j=1}^{k}\Sigma_{r=1}^{\kappa}
	\Sigma _{i=1}^{\eta_{r}} \mu(x^r_{i,j}) + \Sigma_{p=1}^{m} (-1)^p \begin{pmatrix} m \\
		p \end{pmatrix} \Sigma_{j=1}^{m}\Sigma_{r=1}^{\kappa}
	\Sigma _{i=1}^{\eta_{r}} \mu(x^r_{i,k+j}) .$ \\\\\\
	Since $\Sigma_{p=0}^{m} (-1)^p \begin{pmatrix} m \\
		p \end{pmatrix} \Sigma_{j=1}^{k}\Sigma_{r=1}^{\kappa}
	\Sigma _{i=1}^{\eta_{r}} \mu(x^r_{i,j}) =0$, it follows that 
	\begin{align*}
	\Sigma_{r=1}^\kappa \mu(x_r) \Sigma_{p=0}^{m} (-1)^p \begin{pmatrix} m \\
		p \end{pmatrix} h_{p+k}(x_r) &= 0 + \Sigma_{r=1}^{\kappa} \Sigma _{i=1}^{\eta_{r}} \Sigma_{j=1}^{m} \Sigma_{p=j}^{m} (-1)^p \begin{pmatrix} m \\ p \end{pmatrix} \mu(x^r_{i,k+j})\\
	 &= \Sigma_{r=1}^{\kappa} \Sigma _{i=1}^{\eta_{r}} \Sigma_{j=1}^{m}  (-1)^j \begin{pmatrix} m-1 \\ j-1 \end{pmatrix} \mu(x^r_{i,k+j}) \\
	& = -\Sigma_{r=1}^{\kappa} \Sigma _{i=1}^{\eta_{r}} \Sigma_{j=0}^{m-1}  (-1)^j \begin{pmatrix} m-1 \\ j \end{pmatrix} \mu(x^r_{i,k+j+1}) \\
	&= -\Sigma_{r=1}^{\kappa} \Sigma _{i=1}^{\eta_{r}} \Sigma_{j=0}^{m-1}  (-1)^j \begin{pmatrix} m-1 \\ j \end{pmatrix} \mu(x^r_{i,1}) h_{k+j}(x^r_{i,1}) \\
	&= -\Sigma_{r=1}^{\kappa} \Sigma _{i=1}^{\eta_{r}} \Sigma_{p=0}^{m-1}  (-1)^j \begin{pmatrix} m-1 \\ p \end{pmatrix} \mu(x^r_{i,1}) h_{k+p}(x^r_{i,1}).
	\end{align*}
This completes the proof.
%%%%%%%%%%%%%%%%%%%%%%%%%%%%%%%%%%%%%%%%%%%%%%%%%
\end{proof}
\begin{proposition}\label{p1}
	Suppose $m\geq 2$, $k\in\mathbb{Z_+}$,  (\ref{equ1}) holds,  $\{\mu(x^{r}_{i,k+j+1})\}_{j=0}^\infty $ is a polynomial in $j$ of degree at most $m-1$ for every $r\in J_\kappa $, $i\in J_{\eta_{r} }$ and 
	$  \Sigma_{i=1}^{\eta_r} \Sigma_{j=1}^m \mu(x_{i, j}^r) <\infty$  for all $r\in J_\kappa $. Then $\{\mu(x^r_{i,k+j+1})\}_{j=0}^\infty $ is a polynomial in $j$ of degree at most $m-2$  if and only if 
	$	\Sigma_{r=1}^\kappa \mu(x_r) \Sigma_{p=0}^{m} (-1)^p \begin{pmatrix} m \\
		p \end{pmatrix} h_{p+k}(x_r) = 0 .$
\end{proposition}
\begin{proof}
Given that  $\{\mu(x^{r}_{i,k+j+1})\}_{j=0}^\infty $ is a polynomial in $j$ of degree at most $m-1$ for every $r\in J_\kappa $, $i\in J_{\eta_{r} }$ and 	$  \Sigma_{i=1}^{\eta_r} \Sigma_{j=1}^m \mu(x_{i, j}^r) <\infty$ for all  $r\in J_\kappa$. 	 Then by \cite[Corollory 2.2 ]{ZJJK }, we have
 $$ \Delta^{m-1}(\mu(x^{r}_{i,k+j+1})) = a_{i}^r, r\in J_\kappa , i\in J_{\eta_{r} }. $$
 Then by Lemma \ref{l6} , we obtain
 \begin{align*}
	\Sigma_{r=1}^\kappa \mu(x_r) \Sigma_{p=0}^{m} (-1)^p \begin{pmatrix} m \\
	 	p \end{pmatrix} h_{p+k}(x_r)  &
	 = -\Sigma_{r=1}^{\kappa} \Sigma _{i=1}^{\eta_{r}} \Sigma_{j=0}^{m-1}  (-1)^j \begin{pmatrix} m-1 \\ j \end{pmatrix} \mu(x^r_{i,k+j+1})  \\
	 &  = -\Sigma_{r=1}^{\kappa} \Sigma _{i=1}^{\eta_{r}} (-1)^{m-1} \Delta^{m-1} (\mu(x^r_{i,k+j+1}))_0 \\
	   &= -\Sigma_{r=1}^{\kappa} \Sigma _{i=1}^{\eta_{r}} (-1)^{m-1} a_{i}^r .
\end{align*}	  
	  Since $ a_i^{r}> 0$ for all $~~r\in J_\kappa $ ,  $i\in J_{\eta_r} $ and $j \in\mathbb{Z_+}$,  it follows that  
	  \begin{align*}
	   \Sigma_{r=1}^\kappa \mu(x_r) \Sigma_{p=0}^{m} (-1)^p \begin{pmatrix} m \\
	  	p \end{pmatrix} h_{p+k}(x_r)= 0 
  ~~	& \iff a_i^{r}= 0\\
~~	& \iff\Delta^{m-1} (\mu(x^r_{i,k+j+1})) = 0.
	\end{align*}

That is, 
	$   \{\mu(x^{r}_{i,k+j+1})\}_{j=0}^\infty $ is a polynomial in $j$ of degree at most  $m-2$  for every  $r\in J_\kappa $ and $ i\in J_{\eta_{r} }.$
	\end{proof}
	
\begin{theorem}\label{thm1}
	Let $m\geq 2$, $k\in\mathbb{Z_+}$, and  (\ref{equ1}) holds. Then $ C \in B(L^2(\mu))$ is $k$-quasi-$m$-isometry  if and only if $ \{\mu(x^{r}_{i,k+j+1})\}_{j=0}^\infty $ is a polynomial in $j$ of degree at most $m-2$ for every $r\in J_\kappa $, $i\in J_{\eta_{r} }$ and  $	\Sigma_{p=0}^{m} (-1)^p \begin{pmatrix} m \\
		p \end{pmatrix} h_{p+k}(x_r)= 0  ~~ \textrm{for all}~~r \in J_\kappa .$
\end{theorem}

\begin{proof}
	For $m\geq 2$, $k\in\mathbb{Z_+}$, the composition operator induced by the measurable function $\phi$ on  directed graphs with one circuit and more than one branching vertex, $ C \in B(L^2(\mu))$ is $k$-quasi-$m$-isometry  if and only if 
	$$\Sigma_{p=0}^{m} (-1)^p \begin{pmatrix} m \\
		p \end{pmatrix} h_{p+k}(x)= 0 $$
		$~~\textrm{for all} ~~ x \in X$. That is, 
	\begin{equation}\label{eqn2}
	\Sigma_{p=0}^{m} (-1)^p \begin{pmatrix} m \\
		p \end{pmatrix} h_{p+k}(x_r)= 0 
	\end{equation} 
	$\textrm{for all}  ~~r \in J_\kappa$ and 
\begin{equation}\label{eqn3}
	\Sigma_{p=0}^{m} (-1)^p \begin{pmatrix} m \\
		p \end{pmatrix} h_{p+k}(x^r_{i,j})= 0 ~~ \textrm{for all} ~~r \in J_\kappa, i\in J_{\eta_{r} }, j\in \mathbb{N}.
\end{equation}
	From (\ref{eqn2}), we get $\{\mu(x^{r}_{i,k+j+1})\}_{j=0}^\infty $ is a polynomial in $j$ of degree at most $m-1$  for every $r\in J_\kappa $, $i\in J_{\eta_{r} }$. From (\ref{equ1}) and  Proposition \ref{p1} , it follows  that 
	$ C \in B(L^2(\mu))$ is $k$-quasi-$m$-isometry  if and only if 
	$\{\mu(x^{r}_{i,k+j+1})\}_{j=0}^\infty $ is polynomial in $j$ of degree at most $m-2$ for every $r\in J_\kappa $, $i\in J_{\eta_{r} }$ and  $	\Sigma_{p=0}^{m} (-1)^p \begin{pmatrix} m \\
		p \end{pmatrix} h_{p+k}(x_r)= 0 ~~ \textrm{for all} ~~r \in J_\kappa.$
	\end{proof}
	
%%%%%%%%%%%%%%%%%%%%%%%%%%%%%%%%%%%%%%%%%%%%%%%%%%%%%
	
\begin{corollary}\label{cr1}
	If $\kappa=1 $ in (\ref{equ1}) condition, $m\geq 2$, then $ C \in B(L^2(\mu))$ is $k$-quasi-$m$-isometry  if and only if 
	$\{\mu(x^{1}_{i,k+j+1})\}_{j=0}^\infty $ is a polynomial in $j$ of degree at most $m-2$ for every  $i\in J_{\eta_{1} }$. Moreover,  if at least one of the sequence 	$\{\mu(x^{1}_{i,k+j+1})\}_{j=0}^\infty $ is a polynomial in $j$ of degree $m-2$ for some $i\in J_{\eta_{1} }$, then $C$ is strict $k$-quasi-$m$-isometry. 
\end{corollary}
\begin{proof}
	Assume that $\kappa=1 $ in (\ref{equ1}), $m\geq 2$ and $ C \in B(L^2(\mu))$.
	Then by  Theorem \ref{thm1} and Proposition \ref{p1},  it is clear that  $ C \in B(L^2(\mu))$ is $k$-quasi-$m$-isometry  if and only if 
	$\{\mu(x^{1}_{i,k+j+1})\}_{j=0}^\infty $ is a polynomial in $j$ of degree at most $m-2$, for every  $i\in J_{\eta_{1} }$. For the second part, if at least one of the sequence 	$\{\mu(x^{1}_{i,k+j+1})\}_{j=0}^\infty $ is a polynomial in $j$ of degree $m-2$ for some $i\in J_{\eta_{1} }$, then $C$ is not  $k$-quasi-$n$- isometry for any $n<m$. Therefore, $C$ is strict $k$-quasi-$m$- isometry.
\end{proof}
%%%%%%%%%%%%%%%%%%%%%%%%%%%%%%%%%%%%%%%%%%%%%%%%%%%%%
\begin{corollary}\label{cr2}
	Let $m\geq 2$, $k=0$, and (\ref{equ1}) holds. Then $ C \in B(L^2(\mu))$ is $m$-isometry  if and only if 	$ \{\mu(x^{r}_{i,j+1})\}_{j=0}^\infty $ is  a polynomial in $j$ of degree at most $m-2$ for every $r\in J_\kappa $, $i\in J_{\eta_{r} }$ and  $	\Sigma_{p=0}^{m} (-1)^p \begin{pmatrix} m \\
		p \end{pmatrix} h_{p}(x_r)= 0 ~~ \textrm{for all} ~~r \in J_\kappa. $
\end{corollary}
\begin{proof}
If $k=0$., then the required result follows by Theorem \ref{thm1}.
\end{proof}
\begin{corollary}(\cite[Theorem 2.11]{ZJJK })\label{cr3}
Let $m\geq 2$, $k=0$,  $\eta_i= 0,~~ \textrm{for all} ~~ i \in J_{\kappa -1}$ and (\ref{equ1}) hold. Then $ C \in B(L^2(\mu))$ is $m$- isometry if and only if $ \{\mu(x^{\kappa}_{i,j+1})\}_{j=0}^\infty $ is a polynomial in $j$ of degree atmost $m-2$ for every  $i\in J_{\eta_{\kappa} }$ and  $	\Sigma_{p=0}^{m} (-1)^p \begin{pmatrix} m \\
	p \end{pmatrix} h_{p}(x_r)= 0 ~~ \textrm{for all} ~~r \in J_\kappa $.
\end{corollary}

\begin{proof}
	Given that $m\geq 2$, $k=0$,  $\eta_i= 0~~ \textrm{for all} ~~ i \in J_{\kappa -1}$ and (\ref{equ1})  holds. Then  the required result follows by Corollary \ref{cr2}.

\end{proof}

\begin{example}\label{e1}
	Let $\kappa=3$, $\eta_1=2$, $\eta_2= \eta_3 =0$,  $k=1$, $m=2$ and   (\ref{equ1})  hold. Define $ \mu(x^{r}_{i,k+j+1})= \mu(x^{r}_{i,j+2})= 1$ for $r \in J_\kappa, ~~i\in J_{\eta_r}, \mathrm{and}~~ j \in \mathbb{Z_+} $. If   $ \mu(x_1)=\frac{5}{3}, \mu(x_2)=\frac{1}{3}, \mu(x_3)=1, \mu(x_{i,j}^1)=1,$  for 
	$  i \in J_{\eta_1}, j=1$,  
	then  we have
	\begin{align*}
		\begin{split}
			\mu(x_2)+ \Sigma_{i=1}^{\eta_1}\mu(x_{i, 1}^{1})
			-2 [\mu(x_3)+ \Sigma_{i=1}^{\eta_1}\mu(x_{i, 2}^{1})]
			+\mu(x_1)+ \Sigma_{i=1}^{\eta_1}\mu(x_{i, 3}^{1}) = 0,
			\end{split}
		\end{align*}
		\begin{align*}
		\begin{split}
			\mu(x_3)
		-2 \mu(x_1)
		+\mu(x_2)+ \Sigma_{i=1}^{\eta_1}\mu(x_{i, 1}^{1})= 0,
		\end{split}
	\end{align*}
	and	\begin{align*}
		\begin{split}
			\mu(x_1)
			-2 [\mu(x_2)+ \Sigma_{i=1}^{\eta_1}\mu(x_{i, 1}^{1})]
			+\mu(x_3)+ \Sigma_{i=1}^{\eta_1}\mu(x_{i, 2}^{1}) = 0.
		\end{split}
	\end{align*}

 Then by Theorem \ref{thm1}, the composition operator $C$ is quasi-2-isometry.

	\end{example}

\begin{example}
	Let $\kappa=3$, $\eta_1=2$, $\eta_2=1, \eta_3 =0$,  $k=2$, $m=2$, and  (\ref{equ1})  hold. 	Define $ \mu(x^{r}_{i,k+j+1})= \mu(x^{r}_{i,j+3})= 1$ for $r \in J_\kappa, ~~i\in J_{\eta_r},\mathrm{and} ~~ j \in \mathbb{Z_+} $. If  $ \mu(x_1)=2, \mu(x_2)=1, \mu(x_3)=1, \mathrm{and}~~~ \mu(x_{i,j}^r)=1$  for 
	$ r\in J _\kappa, i \in J_{\eta_r}, j=1, 2$,  
	then  
	\begin{align*}
		\begin{split}
			\mu(x_3)+ \Sigma_{i=1}^{\eta_1}\mu(x_{i, 2}^{1})+\Sigma_{i=1}^{\eta_2}\mu(x_{i, 1}^{2})
			-3 [\mu(x_1)+ \Sigma_{i=1}^{\eta_1}\mu(x_{i, 3}^{1})+ \Sigma_{i=1}^{\eta_2}\mu(x_{i, 2}^{2})]\\
			+3[\mu(x_2)+ \Sigma_{i=1}^{\eta_1}[\mu(x_{i, 4}^{1})+ \mu(x_{i, 1}^{1})]+\Sigma_{i=1}^{\eta_2}\mu(x_{i, 3}^{2})]\\
			-[\mu(x_3)+ \Sigma_{i=1}^{\eta_1}[\mu(x_{i, 5}^{1})+ \mu(x_{i, 2}^{1})]+\Sigma_{i=1}^{\eta_2}[\mu(x_{i, 4}^{2})+\mu(x_{i, 1}^{2})]]= 0,
		\end{split}
	\end{align*}
	\begin{align*}
		\begin{split}
			\mu(x_1)+ \Sigma_{i=1}^{\eta_2}\mu(x_{i, 2}^{2})
			-3 [\mu(x_2)+ \Sigma_{i=1}^{\eta_1}\mu(x_{i, 1}^{1})+ \Sigma_{i=1}^{\eta_2}\mu(x_{i, 3}^{2})]\\
			+3[\mu(x_3)+ \Sigma_{i=1}^{\eta_1}\mu(x_{i, 2}^{1}) +\Sigma_{i=1}^{\eta_2}[\mu(x_{i, 4}^{2})+\mu(x_{i, 1}^{2})]]\\
			-[\mu(x_1)+ \Sigma_{i=1}^{\eta_1}\mu(x_{i, 3}^{1})+\Sigma_{i=1}^{\eta_2}[\mu(x_{i, 5}^{2})+\mu(x_{i, 2}^{2})]] = 0,
		\end{split}
	\end{align*}
	and \begin{align*}
		\begin{split}
			\mu(x_2)+ \Sigma_{i=1}^{\eta_1}\mu(x_{i, 1}^{1})
			-3 [\mu(x_3)+ \Sigma_{i=1}^{\eta_1}\mu(x_{i, 2}^{1})+ \Sigma_{i=1}^{\eta_2}\mu(x_{i, 1}^{2})]\\
				+3[\mu(x_1)+ \Sigma_{i=1}^{\eta_1}\mu(x_{i, 3}^{1}) +\Sigma_{i=1}^{\eta_2}\mu(x_{i, 2}^{2})]\\
					-[\mu(x_2)+ \Sigma_{i=1}^{\eta_1}[\mu(x_{i, 4}^{1})+\mu(x_{i, 1}^{1})]+\Sigma_{i=1}^{\eta_2}\mu(x_{i, 3}^{2})]= 0.
		\end{split}
	\end{align*}
	
	Then,  $C$ is  2-quasi-2-isometry.
	
\end{example}
%%%%%%%%%%%%%%%%%%%%%%%%%%%%%%%%%%%%%%%%%%%%%%%%%
\textbf{Weighted composition operators:}   Let $\pi \in  L^{\infty}(\mu)$ and $\phi$ be a nonsingular measurable transformation defined on (\ref{equ1}).   Then for any $p \in \mathbb{N}$, define
  $$(\pi \circ \phi^{p})(x)= \left\{
  \begin{array}{ll}
  	\pi(x^r_{i,j})& \mathrm{if}~~ x=x^r_{i,j+p} ~~~\mathrm{for}~~r\in J_\kappa, ~~ i\in J_{\eta_r},~~ \mathrm{and}~~j\in \mathbb{N} ,\\\\
  	\pi(x_r)&\mathrm{if}~~ x=x^s_{i,j} ~~~\mathrm{for}~~ s\in J_\kappa,  \mathrm{and}~~ \Phi_2(p+r)=\Phi_2(s+j),~~  j\in J_p, \\
  	& ~~i\in J_{\eta_s}, \mathrm{or}~~x=x_{\Phi_2(p+r)},\\\\	
  \end{array}\right.$$
  and
  \begin{align*}
  \pi^2_p(x)=\pi^2(x)(\pi \circ \phi)^2(x)(\pi \circ \phi^{2})^2(x)\ldots (\pi \circ \phi^{p-1})^2(x).
  \end{align*}
  For $p\geq\kappa$,  we obtain $E_p(\pi^2_p)$  by using atoms of $\phi^{-p}(\mathcal{F})$  as follows:
  $$E_p(\pi^2_p)(x)= \left\{
  \begin{array}{ll}
  	K^r_{i, j+p}& \mathrm{if}~~ x=x^r_{i,j+p} ~~~\mathrm{for}~~ r\in J_\kappa ~~i\in J_{\eta_r},~~ \mathrm{and}~~j\in \mathbb{N} ,\\\\
  	K_p^r& \mathrm{if}~~ x=x^s_{i,j} ~~~\mathrm{for}~~ s\in J_\kappa,  \mathrm{and}~~ \Phi_2(p+r)=\Phi_2(s+j),~~  j\in J_p, \\
  	& ~~ i\in J_{\eta_s},\mathrm{or}~~x=x_{\Phi_2(p+r)},\\	
  \end{array}\right.$$
  where $K^r_{i, j+p}= \pi_p^2(x^r_{i, j+p})$ 
  and 
  
 $$K_p^r= \frac{\pi_p^2(x_{{\Phi_2}_(p+r)})\mu(x_{{\Phi_2}_(p+r)})+\Sigma_{j=1}^p \Sigma_{s=1,  \Phi_2(p+r)=\Phi_2(s+j)}^\kappa \Sigma_{i=1}^{\eta_s}\pi_p^2(x^s_{i,j})\mu(x^s_{i,j})}{\mu(x_{{\Phi_2}_(p+r)})+\Sigma_{j=1}^p \Sigma_{s=1, \Phi_2(p+r)=\Phi_2(s+j)}^\kappa \Sigma_{i=1}^{\eta_s}\mu(x^s_{i,j})}.$$
 
Since the  conditional expectation $E_p(\pi^2_p)$ is a $\phi^{-p}(\mathcal{F})$-measurable function on $X$, there exist a $\mathcal{F}$-measurable function $F_p$ on $X$ such that $E_p(\pi^2_p)=F_P\circ \phi^p$, where $F_P$ can be defined as follows:
  $$F_p(x)= \left\{
  \begin{array}{ll}
  	K^r_{i, j+p}& \mathrm{if}~~ x=x^r_{i,j} ~~~\mathrm{for}~~ r\in J_\kappa,~~ i\in J_{\eta_r},~~ \mathrm{and}~~j\in \mathbb{N} ,\\\\
  	K^r_p& \mathrm{if}~~ x=x_r ~~~\mathrm{for}~~ r\in J_\kappa.\\	
  \end{array}\right.$$
Now  we have  $W^{*p}W^p = h_pE_p(\pi_{p}^{2})\circ \phi ^{-p} = h_pF_p.$
Then for any $m \in \mathbb{N},$  \\
$$\mathcal{B}_m(W) = \Sigma_{p=0}^m (-1)^p \begin{pmatrix}
	m\\p
\end{pmatrix}W^{*p}W^p = \Sigma_{p=0}^m (-1)^p \begin{pmatrix}
m\\p
\end{pmatrix}h_pF_p .$$

The following lemma  is immediate  from Lemma \ref{l1} and the generalization of  \cite[Theorem 3.2]{SPD } for $k$-quasi-$m$-isometric weighted composition operators.

\begin{lemma} \label{l7}Let  $(X,  \mathcal{F}, \mu)$ be a discrete measure space, $\phi$ be a nonsingular measurable transformation on $X$, $\pi \in  L^{\infty}(\mu)$,  and let   $W$ be the weighted  composition operator on $L^{2}(\mu)$ induced by  $\phi$ and $\pi$. Then for any $m\in\mathbb{N}$ and $k\in\mathbb{Z_+}$ the following are equivalent:\\\\
	(i) $W$ is an $k$-quasi-$m$-isometry,\\\\
	(ii) $\sum^{m}_{p=0} (-1)^{p}\begin{pmatrix} m \\
		p \end{pmatrix} W^{*(k+p)}W^{(k+p)}=0,$\\\\
	(iii) $\sum^{m}_{p=0} (-1)^{p}\begin{pmatrix} m \\
		p \end{pmatrix} W^{*(n+k+p)}W^{(n+k+p)}=0,~~~\mathrm{for}~~~n\in\mathbb{Z_+},$\\\\
	(iv) $\sum^{m}_{p=0} (-1)^{p}\begin{pmatrix} m \\
		p \end{pmatrix} h_{n+k+p}F_{n+k+p}(x) =0~~ \textrm{for all} ~~ x\in X~~~\mathrm{and}~~~n\in\mathbb{Z_+},$\\\\
	(v) $\{h_{n+k}F_{n+k}(x)\}_{n=0}^\infty$ is a polynomial in $n$ of degree at most $m-1$ for all $x \in X$.	
\end{lemma}

%%%%%%%%%%%%%%%%%%%%%%%%%%%%%%%%%%%%%%%%%%%%%%%%%%%%

\begin{theorem}\label{thm3}
	Let $m\geq 2$, $k\in\mathbb{Z_+}$, and (\ref{equ1})  hold. Then the weighted composition operator  $ W \in \mathcal{B}(L^2(\mu))$ induced by  $\phi$ and $\pi$ is $k$-quasi-$m$- isometry 	 if and only if $  \{\pi_k^2(x^{r}_{i,k+j+1})\mu(x^{r}_{i,k+j+1})\}_{j=0}^\infty $ is a polynomial in $j$ of degree at most $m-1$ for every $r\in J_\kappa $, $i\in J_{\eta_{r} }$ and  
	 $	\Sigma_{p=0}^{m} (-1)^p \begin{pmatrix} m \\
			p \end{pmatrix} h_{p+k}F_{p+k}(x_r)= 0 ~~ \textrm{for all} ~~r \in J_\kappa .$
	\end{theorem}
	
\begin{proof}
By Lemma \ref{l7} , $W$ is  $k$-quasi-$m$-isometry if and only if \\
$$\sum^{m}_{p=0} (-1)^{p}\begin{pmatrix} m \\
	p \end{pmatrix} h_{p+k}F_{p+k}(x) =0, $$ 
	$~~~\textrm{for all}~~~ x\in X$. That is , $ ~~\text{for all}~~ r\in J_\kappa, , j \in \mathbb{N}, i \in J_{\eta_{r}}$
$$\sum^{m}_{p=0} (-1)^{p}\begin{pmatrix} m \\
	p \end{pmatrix} h_{p+k}F_{p+k}(x_{i,j}^r) =0$$
and  $~~\text{for all}~~ r\in J_\kappa$
$$\sum^{m}_{p=0} (-1)^{p}\begin{pmatrix} m \\
	p \end{pmatrix} h_{p+k}F_{p+k}(x_r) =0.$$

 Thus, 
 $ \{\pi_k^2(x^{r}_{i,k+j+1})\mu(x^{r}_{i,k+j+1})\}_{j=0}^\infty $ is a  polynomial in $j$ of degree at most $m-1$ for every $r\in J_\kappa $, $i\in J_{\eta_{r} }$ and 
 $	\Sigma_{p=0}^{m} (-1)^p \begin{pmatrix} m \\
 	p \end{pmatrix} h_{p+k}F_{p+k}(x_r)= 0  ~~\text{for all}~~r \in J_\kappa $. 
 	This completes the proof.
\end{proof}
%%%%%%%%%%%%%%%%%%%%%%%%%%%%%%%%%%%%%%%%%%%%%%%%%%%%

\begin{corollary}\label{cr5}
	Let $m\geq 2$, $k=0$, and (\ref{equ1}) holds. Then the weighted composition operator  $ W \in \mathcal{B}(L^2(\mu))$ induced by  $\phi$ and $\pi$ is $m$-isometry 	 if and only if $  \{\pi^2(x^{r}_{i,j+1})\mu(x^{r}_{i,j+1})\}_{j=0}^\infty $ is a polynomial in $j$ of degree at most $m-1$ for every $r\in J_\kappa $, $i\in J_{\eta_{r} }$ and 
	$	\Sigma_{p=0}^{m} (-1)^p \begin{pmatrix} m \\
		p \end{pmatrix} h_{p}F_{p}(x_r)= 0  ~~ \mathrm{for~~ all}  ~~r \in J_\kappa .$

\end{corollary}
%%%%%%%%%%%%%%%%%%%%%%%%%%%%%%%%%%%%%%%%%%%%%%%%%%%%
\begin{corollary}\label{cr6}
	Let $m\geq 2$,   $\pi=1$ and  (\ref{equ1}) holds. Then the weighted composition operator  $ W \in B(L^2(\mu))$ induced by  $\phi$ and $\pi$ is $k$-quasi-$m$-isometry 	 if and only if $  \{\mu(x^{r}_{i,k+j+1})\}_{j=0}^\infty $ is a  polynomial in $j$ of degree at most $m-2$ for every $r\in J_\kappa $, $i\in J_{\eta_{r} }$ and 
	$	\Sigma_{p=0}^{m} (-1)^p \begin{pmatrix} m \\
		p \end{pmatrix} h_{p+k}(x_r)= 0  ~~\text{for all} ~~r \in J_\kappa .$
	
\end{corollary}
\begin{proof}
	Assume that $m\geq 2$, ,  $\pi=1$ and (\ref{equ1}) holds. 
	Since $\pi = 1 $, $\pi_{p+k}^2 (x) = 1, ~~ \text{for all} ~~ x\in X $ and  $F_{p+k}(x)=1,~~ \text{for all} ~ x\in X, p \in \mathbb{N}, k\in \mathbb{Z_+}$. Then by  Theorem \ref{thm3} and Corollary \ref{cr2} , it follows that 
	$ W \in B(L^2(\mu))$ is $k$-quasi-$m$-isometry 	 if and only if $  \{\mu(x^{r}_{i,k+j+1})\}_{j=0}^\infty $ is a polynomial in $j$ of degree at most $m-2$ for every $r\in J_\kappa $, $i\in J_{\eta_{r} }$ and 
	$	\Sigma_{p=0}^{m} (-1)^p \begin{pmatrix} m \\
		p \end{pmatrix} h_{p+k}(x_r)= 0 ~~ \text{for all}  ~~r \in J_\kappa $.
\end{proof}

\begin{example}\label{e3}
Let $\kappa=3$, $\eta_1=2$, $\eta_2= \eta_3 =0$,  $k=1$, $m=2$, and  (\ref{equ1})  hold. Define  $\pi_k^2(x^{r}_{i,k+j+1}) \mu(x^{r}_{i,k+j+1})= \pi_k^2(x^{r}_{i,j+2})\mu(x^{r}_{i,j+2})= 1$ for $r \in J_\kappa, ~~i\in J_{\eta_r}, ~~ j \in \mathbb{Z_+}, $ and
	$$ \pi(x)= \left\{
	\begin{array}{ll}
		\frac{1}{2} & \mathrm{if}~~ x=x^r_{i,j} ~~~\mathrm{for}~~r\in J_\kappa, ~~ i\in J_{\eta_r},~~ \mathrm{and}~~j\in \mathbb{N} ,\\\\
		1&\mathrm{if}~~ x=x_r, ~~r\in J_\kappa. \\	
	\end{array}\right.$$
	If we take  $ \mu(x_1)=\frac{31}{32}, \mu(x_2)=\frac{11}{12}, \mu(x_3)=1, \mu(x_{1,1}^1)=1,\mu(x_{2,1}^1)=\frac{1}{3} $, 
then 
\begin{align*}
	\begin{split}
		\pi^2(x_2)\mu(x_2)+ \Sigma_{i=1}^{\eta_1}\pi^2(x_{i, 1}^{1})\mu(x_{i, 1}^{1})	
		-2 [\pi_2^2(x_3)\mu(x_3)+ \Sigma_{i=1}^{\eta_1}\pi_2^2(x_{i, 2}^{1})\mu(x_{i, 2}^{1})]\\
		+ \pi_3^2(x_1)\mu(x_1)+ \Sigma_{i=1}^{\eta_1}\pi_3^2(x_{i, 3}^{1})\mu(x_{i, 3}^{1}) = 0,
	\end{split}
\end{align*}
\begin{align*}
	\begin{split}
		\pi^2(x_3)\mu(x_3)
		-2 \pi_2^2(x_1)\mu(x_1)
		+ \pi_3^2(x_2)\mu(x_2)+ \Sigma_{i=1}^{\eta_1}\pi_3^2(x_{i, 1}^{1})\mu(x_{i, 1}^{1}) = 0,
	\end{split}
\end{align*}
and \begin{align*}
	\begin{split}
		\pi^2(x_1)\mu(x_1)
		-2 [\pi_2^2(x_2)\mu(x_2)+ \Sigma_{i=1}^{\eta_1}\pi_2^2(x_{i, 1}^{1})\mu(x_{i, 1}^{1})]\\
		+ \pi_3^2(x_3)\mu(x_3)+ \Sigma_{i=1}^{\eta_1}[\pi_3^2(x_{i, 2}^{1})\mu(x_{i, 2}^{1}) = 0.
	\end{split}
\end{align*}
Then by Theorem \ref{thm1}, $W$ is  quasi-2-isometry.
\end{example}

\noindent \textit{Acknowledgments}. \  The third-named author was supported in part by the Mathematical Research Impact Centric Support, MATRICS (MTR/2021/000373) by SERB,
Department of Science and Technology (DST), Government of India.

\end{document}